\documentclass[preprint,12pt]{elsarticle} 
\usepackage{amsmath}
\usepackage{amssymb} \usepackage{amsthm} \usepackage{enumerate}
\usepackage[utf8]{inputenc}
\usepackage{calrsfs}
\usepackage{dsfont}
\usepackage{cancel}
\usepackage[english]{babel}
\usepackage{hyperref}
\usepackage{multirow}
\usepackage{hhline}

\newtheorem{Theorem}{Theorem}[section]
\newtheorem{Corollary}[Theorem]{Corollary}
\newtheorem{Lemma}[Theorem]{Lemma}
\newtheorem{Example}[Theorem]{Example}

\newtheorem{Proposition}[Theorem]{Proposition}
\theoremstyle{definition}
\newtheorem{Definition}[Theorem]{Definition}
\newtheorem{Notation}[Theorem]{Notation}

\newtheorem{Problem}[Theorem]{Problem}
\newtheorem{Remark}[Theorem]{Remark}

\numberwithin{equation}{section}

\def\A{{\mathcal A}}
\def\B{{\mathcal B}}
\def\N{{\mathcal N}}
\def\V{{\mathcal V}}
\def\W{{\mathcal W}}
\def\Z{{\mathcal Z}}
\def\CC{{\mathcal C}}

\def\KKK{{\bf K}}
\def\C{{\mathbb C}}
\def\R{{\mathbb R}}
\def\HH{{\mathbb H}}
\def\NN{{\mathbb N}}  

\def\PG{{$\mathcal{PG}$}}
\def\MPG{{$\mathcal{MPG}$}}

\def\mgen{{\mathrm  {mgen}}}
\def\diag{{\mathrm {diag}}}
\def\row{{\mathrm {row}}}
\def\col{{\mathrm {col}}}

\def\1{{\mathds{1}}}

\newcommand{\gen}[1]{\ensuremath {\langle #1\rangle _{Alg}}}
\newcommand{\lin}[1]{\ensuremath {\langle #1\rangle}}

\newcommand{\w}[1]{\widetilde{#1}}
\newcommand{\h}[1]{\widehat{#1}}

\begin{document}
\begin{frontmatter}
	
\title{On real algebras generated by positive and nonnegative matrices}

\author[MSU,CFAM]{N.A.~Kolegov\fnref{fn1}}
\ead{na.kolegov@yandex.ru}

\fntext[fn1]{The work is supported by Foundation for the Advancement of Theoretical Physics and Mathematics “BASIS”, grant 19-8-2-33-1.}

\address[MSU]{Lomonosov Moscow State University, Moscow, 119991, Russia.}
\address[CFAM]{Moscow Center for Fundamental and Applied Mathematics, Moscow, 119991, Russia.}

\begin{abstract}
{ Algebras generated by strictly positive matrices are described up to similarity, including the commutative, simple, and semisimple cases. We provide sufficient conditions for some block diagonal matrix algebras  to be generated by a set of nonnegative matrices up to similarity. Also we find all realizable dimensions of algebras generated by two nonnegative semi-commuting matrices. The last result provides the solution to the problem posed by M.~Kandi\'{c}, K.~\u{S}ivic (2017) \cite{KanSiv17}.}
\end{abstract}


\end{frontmatter}

\section{Introduction}

Two real matrices $A,B$ are called semi-commuting if their commutator $[A,B]=AB-BA$ is a nonpositive or nonnegative matrix. Algebras ge\-ne\-rated by two nonnegative semi-commuting matrices were considered first in \cite{Drnov, KanSiv17}. It was shown that dimensions of such algebras do not exceed $\frac{n(n+1)}{2}$ \cite[Theorem 3.2]{KanSiv17}. At the same time the following problem was posed.

\begin{Problem}{\cite[Question 3.4]{KanSiv17}} \label{problem}
	 Let $n\in\NN$, $n\leq k\leq \frac{n(n+1)}{2}$. Do there exist semi-commuting nonnegative matrices $A,B$ such that the unital algebra generated by $A,B$ has the dimension $k$?
\end{Problem}
To avoid an ambiguity, we note that the  paper \cite{KanSiv17} uses another terminology. Nonnegative and positive matrices are called positive and strictly positive, respectively.
It turns out that the answer to the above question is affirmative, the complete solution is obtained below in Corollary \ref{solution}. Besides, we consider a more general case of algebras generated by nonnegative and positive matrices. It is natural to study such algebras up to an automorphism of $M_n(\R)$, that is, up to similarity due to the Skolem–Noether theorem \cite[Sec. 12.6]{Pierce82}.

The present work is intended to facilitate a better understanding of difficult interrelations between nonnegative matrices and similarities. In particular, is a given matrix similar to nonnegative one? The answer is known in some specific cases. See \cite{BoroMoro97, Kell71} for details. A related task is to determine possible values of invariants of nonnegative matrices, such as spectrum. This is the widely known nonnegative inverse eigenvalue problem which is of particular importance. See \cite{EgLenNa04,JohnMPP18} for reviews of the problem and works \cite{AlPS18, CollJS18, CroLaff18, EllSm16} for some recent results. The present paper treats a generalization of the aforementioned problems. We change a matrix to a matrix algebra. So, the following natural questions arise. Is a given real matrix algebra similar to an algebra generated by nonnegative matrices? Are there any ne\-ces\-sa\-ry or sufficient spectral conditions on matrices of the algebra? If we deal with an algebra instead of a single matrix, it is possible to consider linear combinations and products of matrices. So, the theory of finite-dimensional algebras can be applied. Also we continue investigations of generators of matrix algebras, see works \cite{GutLMS18, KolMar, Pas19} for some recent results on this topic.

The work is organized as follows. In Section \ref{2} a system of notations and some preliminary results are given. In Section \ref{3} a criterion for unital matrix algebras to be positively generated up to si\-mi\-la\-ri\-ty is obtained. Commutative, simple, and semisimple matrix algebras of such type are completely described up to similarity. Section \ref{4} consists of two parts. Subsection \ref{4.1} provides some examples of nonnegatively generated algebras up to si\-mi\-la\-ri\-ty and algebras without this property. In Subsection \ref{4.2} we show that the property to be (minimally) nonnegatively generated up to similarity is preserved after a direct sum with an arbitrary matrix algebra. It gives new sufficient conditions for centralizers and algebras with a non-trivial center to be nonnegatively generated up to similarity. In Section \ref{5} we prove that matrix incidence algebras are generated by two nonnegative semi-commuting matrices. As the consequence, the solution to Problem \ref{problem} is obtained.

\section{Preliminaries}\label{2}

\subsection{Matrix algebras. General facts and constructions}

Let $M_n(\R)$ denote the algebra of all  $n\times n$ matrices over the field $\R$. We always assume $n\geq 1$, unless otherwise stated. Let $T_n(\R)$, $\w{T}_n(\R)$, and $D_n(\R)$ be the subalgebras of upper-triangular, lower-triangular, and diagonal matrices, respectively. 

Also $E_{ij}\in M_n(\R)$ denotes the matrix unit which contains $1$ in the position $(i, j)$ and zeros elsewhere. Besides, $O_n,I_n,\1_n\in M_n(\R)$ denote the identity matrix, the zero matrix, and the matrix of ones, respectively. The zero $m\times n$ matrix is denoted by $O_{m\times n}$. We write $O,I,\1$ without the subscripts if they are known from the context. Denote by $J_k(\lambda)$ the Jordan cell of the size $k$ corresponding to the eigenvalue $\lambda$. 

For $A\in M_n(\R)$,  $(A)_{ij}$ or $a_{ij}$ is a real number in the position $ (i, j)$. The row vector $\row_i(A)$ and the column vector $\col_j(A)$ coincide with the $i$-th row and the $j$-th column of the matrix $A$, respectively. Let $\sigma(A)\subseteq\C$ be the spectrum of the matrix $A$. An eigenvalue is called simple if it has algebraic multiplicity $1$. Let $\rho(A)=\max\{|\lambda|:\lambda\in\sigma(A)\}$ be the spectral radius of $A$. Besides, $||\cdot||$ denotes the uniform norm on matrices, that is, $||A||=\max\{|a_{ij}|:i,j=1,\ldots,n\}$. Also $C(A)$ is a \textit{centralizer (commutant)} of $A$, i.e. $C(A)=\{X\in M_n(\R)~|~AX=XA\}$.

 We say that $R\in M_n(\R)$ has a {\em regular (upper-)triangular form} \cite[Ch. VIII]{Gant60} if there exists a vector $(r_1,\ldots,r_n)\in\R^n$ such that 
 $$
 R=R_{nn}(r_1,\ldots,r_n)=\left(\begin{array}{cccc}
 r_1 & r_2 & \ldots &  r_n \\
 0& r_1 & \ddots& \vdots\\
 \vdots & & \ddots & r_2\\
 0&\ldots &0 &  r_1
 \end{array}\right).
 $$
Let $R$ be a real matrix of a size $p\times q$, $p\neq q$. It has a regular upper-triangular form if 
$$
R=R_{pq}(r_1,\ldots,r_{\min\{p,q\}}) = \left[
\begin{array}{ccc}
\left(O_{p\times(q-p)}~R_{pp}(r_1,\ldots,r_p)\right),~~\text{if}~~ q>p; \\
	\\
\left(\begin{array}{cc}
~R_{qq}(r_1,~\ldots,~r_q)~\\
O_{(p-q)\times q}
\end{array}\right), ~~~~\text{if}~~ p>q.\\
\end{array}
\right.
$$

A matrix algebra $\A\subseteq M_n(\R)$ is called unital if $I_n\in\A$. We will mainly deal with unital matrix algebras. A nonempty set $\Phi \subseteq M_n(\R)$ \textit{generates} a unital matrix algebra $\A$ if $\A$ is the minimal (by inclusion) algebra containing $\Phi\cup\{I_n\}$. The notation is $\A=\gen{\Phi}$. If $\Phi=\{A\}$, we write $\gen{A}$ instead of $\gen{\{A\}}$. Denote $\mgen(\A)=\min\{|\Phi|~|~\gen{\Phi}=\A\}$. A ge\-ne\-rating system of $\A$ is minimal if $|\Phi|=\mgen(\A)$. Let $\lin{\cdot}$ denote the linear span over $\R$.  

If a matrix algebra $\A$ has only trivial two-sided ideals $\{O\}$ and $\A$, it is said to be {\em simple}. A matrix algebra $\A$ is {\em semisimple} if $\A$ is isomorphic, as $\R-$algebra, to a direct sum of simple matrix algebras. However, this isomorphism can differ from similarity.

A matrix algebra $\A\subseteq M_n(\R)$ is called {\em irreducible} if it has exactly two invariant subspaces: $\{0\}$, $\R^n$. An irreducible matrix algebra is always simple \cite[p. 34 Corollary 2.2]{RadRY04} but the converse does not always hold, for instance, consider the algebra $\B=\{A\oplus A~|~A\in M_n(\R)\}\subseteq M_{2n}(\R)$. Next theorem is applicable over any field but we need only the case of $\R$.

\begin{Theorem}[{{\cite{RadRY04, Lam98}} Generalized Burnside's theorem}] \label{burnside}
Let $\A\subseteq M_n(\R)$ be a subalgebra. Then $\A=M_n(\R)$ if and only if $\A$ is irreducible and contains a matrix of rank $1$.
\end{Theorem}

\subsection{Covering matrices and algebras containing all diagonal matrices}
First we introduce the following notation.

\begin{Notation}
	If $\Phi\subseteq M_n(\R)$ is an arbitrary nonempty subset, then
	$$\Omega(\Phi)=\{(i,j)\in\NN\times\NN~|~ \exists A\in\Phi:~a_{ij}\neq 0 \}.$$
\end{Notation}
If $\Phi=\{A\}$ then $\Omega(A)$ will be used instead of $\Omega(\{A\})$. The following relations hold
$$\Omega(\lin{\Phi})=\Omega(\Phi),~\Omega(\Phi)\subseteq\Omega(\gen{\Phi}).$$

\begin{Definition}
	For an arbitrary nonempty set of matrices $\Phi\subseteq M_n(\R)$, a matrix $A\in\Phi$ is said to {\em cover} $\Phi$ or to be a {\em covering matrix} for $\Phi$ if $\Omega(A)=\Omega(\Phi)$.
\end{Definition}
If $\Phi$ is a linear space then  it always contains a covering matrix since the field $\R$ is infinite (see {\cite[Proposition~3.4]{Mar16}}). 

A subalgebra $\A\subseteq M_n(\R)$ is called \textit{matrix incidence algebra} if it has a basis satisfying the following three conditions. $1)$ This basis contains only matrix units. $2)$ All diagonal matrix units are in the basis. $3)$ The basis does not simultaneously contain two symmetric matrix units $E_{ij}$ and $E_{ji}$ for $i\neq j$. For more general approaches, see \cite{Cig19, SpDon}. Note that $$\Omega(\A)=\{(i,j)\in\NN\times\NN~|~E_{ij}\in\A\}.$$ 

Using the same methods as in \cite[Theorems 3.1, 3.2]{KolMar} we obtain the next statement.

\begin{Theorem}[{\cite{KolMar, LonRos00}}] \label{IncGen}
	Let $\A\subseteq M_n(\R)$ be a subalgebra with a covering matrix $A$, $D_n(\R)\subseteq \A$, $D=\diag\{d_1,\ldots,d_n\}$ with $d_i\neq d_j$ if $i\neq j$. Then $\gen{A,D}=\A$.
\end{Theorem}
\begin{proof}
	First note that $\A=\lin{\{E_{ij}~|~(i,j)\in\Omega(\A)\}}$ since $D_n(\R)\subseteq \A$. Furthermore, $D_n(\R)=\gen{D}\subseteq\gen{A,D}$. The equality $\Omega(A)=\Omega(\A)$ implies an inclusion $\gen{D_n(\R),A}\supseteq \lin{\{E_{ij}~|~(i,j)\in\Omega(\A)\}}$. So, ${\gen{A,D}=\A}$.
\end{proof}
In particular, the above theorem is applicable to a matrix incidence algebra. For recent results on algebras containing all diagonal matrices, see \cite{Cig19} and \cite{KolMar}.

Also we need the next result (see \cite[the proof of the main theorem, p. 120]{LonRos00} and \cite[Proposition 1.2.7, p. 15]{SpDon}).

\begin{Theorem}[{\cite{LonRos00, SpDon}}] \label{iso}	
	Let $\A\subseteq M_n(\mathbb{R})$ be a matrix incidence algebra. Then there exist a matrix incidence algebra $\B\subseteq T_n(\mathbb{R})$ and a permutation matrix $P$ such that $\B = P^{-1}\A P$. 
\end{Theorem}

\subsection{The ordered structure on matrix algebras}
Let $\R_{\geq 0}=\R^+=[0,+\infty)$, $\R_{>0}=(0,+\infty)$. The algebra $M_n(\mathbb{R})$ is partially ordered by the following relation. Put $A\geq B$ if and only if $A-B\in M_n(\R^+)$. Also $A>B$ if and only if $A-B\in M_n(\R_{>0})$. Besides, if $A\geq O$ ($A>O$) then $A$ is said to be nonnegative (respectively, positive).
 
 Two real matrices $A$ and $B$ \textit{semi-commute} if the commutator $[A,B] = AB - BA$ is comparable with the zero matrix (either $[A,B]\geq O$, or $[A,B]\leq O$).

Next we formulate some basic properties of algebras generated by nonnegative matrices. 

\begin{Theorem} \label{Criterion}
	Let $\A\subseteq M_n(\R)$ be a unital matrix algebra. The following statements are equivalent.
	\begin{enumerate}
		\item $\A$ has a nonnegative covering matrix. 
		\item $\A$ is generated by a set of nonnegative matrices.
		\item $\A$ has a basis consisting of nonnegative matrices.
	\end{enumerate}
\end{Theorem}
	\begin{proof}
		First we note that $\Omega(\A)\neq\varnothing$ since $\A$ is unital.
		\begin{enumerate}
			\item[$1)\Rightarrow 2)$]  If $\gen{A_1,\ldots, A_m}=\A$ and $A$ is a nonnegative covering matrix, then $\{A_i+\frac{||A_i||}{||A||}A\}_{i=1}^m\cup\{A\}$ is a nonnegative generating system of $\A$.
			\item[$2)\Rightarrow3)$] If $\{A_1,\ldots,A_m\}$ is a nonnegative generating system, then some finite set of products of this matrices constitutes a nonnegative basis.
			\item[$3)\Rightarrow1)$] The sum of the nonnegative matrices from the basis is actually a co\-ve\-ring matrix.
		\end{enumerate}
		
	\end{proof}
	
	Considering the case $\Omega(\A)=\N\times\N$, $\N=\{1,\ldots,n\}$ we immediately get the next corollary.
	
	\begin{Corollary} \label{simplecriterionpositive}
	Let $\A\subseteq M_n(\R)$ be a unital matrix algebra. Then $\A$ is generated by a set of positive matrices if and only if $\A$ contains a positive matrix.
	\end{Corollary}

\begin{Definition} \label{maindef}
	$ $
	\begin{enumerate} 
		\item A subalgebra $\A\subseteq M_n(\R)$ is said to be {\em positively  generated} or a {\em \PG-algebra} if it is generated by a set of positive matrices.
		\item A subalgebra $\A\subseteq M_n(\R)$ is called a {\em \PG-algebra up to similarity} if there exists a nonsingular matrix $C\in M_n(\R)$ such that $C^{-1}\A C$ is a \PG-algebra.
	\end{enumerate}
	
	The similar terminology will be used for algebras that have positive minimal generating systems ({\em minimally positively generated or \MPG-algebras })  and for  algebras generated by nonnegative matrices ({\em nonnegatively generated algebras}). In particular, $M_1(\R)=\R$ is an \MPG-algebra.
\end{Definition}

Note that similarity can preserve the standard order on real matrices only in very few cases. 
\begin{Theorem}[Minc~{\cite{Minc74}}] \label{conjugationnonnegative}
	All automorphisms of $M_n(\R)$ preserving element-wise order are precisely similarities by nonnegative nonsingular monomial matrices.
\end{Theorem}

\subsection{Some technical lemmas}

Here we prove some auxiliary assertions.

\begin{Notation} \label{onlynotation}
Assume that $n\geq 2$. Denote by ${\bf K_n}$ or ${\bf K}$ the following $n\times n$ matrices over $\R$. For $n\geq 3$,\\ {\footnotesize$$\KKK_n^{-1}=\left(\begin{array}{ccccc}
	-1& -1&\ldots &-1 & 1\\
	1& 0&\ldots &0& 1 \\
	0& 1 & \ddots &\vdots &\vdots \\
	\vdots&\vdots&\ddots&0&1\\
	0&0&\ldots&1&1 
	\end{array}\right),~~\KKK_n=\frac{1}{n}\left(\begin{array}{ccccc}
	-1& n-1&-1 &\ldots & -1\\
	-1& -1&n-1 &\ldots& -1 \\
	\vdots& \vdots & \vdots &\ddots &\vdots \\
	-1&-1&-1&\ldots&n-1\\
	1&1&1&\ldots&1 
	\end{array}\right).$$
	$$
	\KKK^{-1}_2=\left(\begin{array}{cc}
	-1& 1\\
	1& 1 
	\end{array}\right),~~\KKK_2=\frac{1}{2}\left(\begin{array}{cc}
	-1& 1\\
	1& 1 
	\end{array}\right).
	$$}
\end{Notation}
The direct calculations show that $\KKK^{-1}_n\KKK_n=I_n$.

\begin{Lemma} \label{idempotentsimilar}
For $\KKK,E_{nn},\1\in M_n(\R)$, we have the identity
${\KKK^{-1}E_{nn}\KKK=\frac{1}{n}\1.}$
Moreover, if $E_{ll}$ is any diagonal matrix unit, then there exists a nonsingular matrix $C\in M_n(\R)$ such that ${C^{-1}E_{ll}C=\frac{1}{n}\1.}$

\end{Lemma}
\begin{proof}
	The first part is proved by the direct calculations. The second part follows from the fact that $\frac{1}{n}\1$ and $E_{ll}$ are idempotents of rank $1$. So, $E_{ll}$ is simply a Jordan normal form of $\frac{1}{n}\1$. Since $E_{ll}$ is a real matrix, the transition matrix $C$ also can be chosen real.
\end{proof}

\begin{Lemma}\label{importantlemma}
	Let $A\in M_n(\R)$, $n\geq 1$, $A=P\oplus Q\oplus R$, where $P\in M_{p}(\R)$, $Q\in M_{q}(\R)$, $R\in M_r(\R)$, $p\geq 0$, $r\geq 0$, $q\geq1$. The equalities $p=0$, $r=0$ mean that the corresponding direct summands are absent. Let $\sigma(Q)\cap(\sigma(P)\cup\sigma(R))=\varnothing$. Then there exists a polynomial $h(x)\in\R[x]$ such that $h(A)=O_p\oplus I_q\oplus O_r$.
	\begin{proof}
		 Let $\mu_{P}(x),\mu_{R}(x)\in\R[x]$ be minimal polynomials of the matrices $P,R$, respectively. Let $f(x)=\mu_{P}(x)\cdot\mu_{R}(x)$. The matrix $f(Q)$ is nonsingular, since $\sigma(Q)\cap(\sigma(P)\cup\sigma(R))=\varnothing$. Applying the Cayley–Hamilton theorem we find a polynomial $g(x)\in\R[x]$ with $g(f(Q))=I_q$ and $g(0)=0$. Choose $h(x)=g(f(x))$.
	\end{proof} 
\end{Lemma}

\section{Positive generating systems up to similarity}\label{3}

This section deals with algebras that are positively generated up to si\-mi\-la\-ri\-ty. Theorem \ref{maintheorem} gives a full description of such algebras. Then we obtain characterizations in commutative (Theorem \ref{commutativematrices}), simple, and semisimple cases (Theorem \ref{semis_theorem}). First the following lemma on one-generated algebras is necessary.

\begin{Lemma} \label{onelemma}
	Let $A\in M_n(\R)$ have a simple real eigenvalue. Then $\gen{A}$ is an \MPG-algebra up to similarity.
\end{Lemma}
\begin{proof}

	We may assume $A$ to equal its real Jordan normal form with an eigenvalue $\lambda_0$ of algebraic multiplicity $1$  in the position $(1,1)$ without loss of generality. Applying Lemma \ref{importantlemma} we find a polynomial $f(x)\in\R[x]$ such that $f(A)=E_{11}$. By Lemma \ref{idempotentsimilar}, there exists $C\in M_n(\R)$ with $C^{-1}E_{11} C=\frac{1}{n}\1$. Consider a matrix $\w{B}=C^{-1} B C$ with $B=A+\beta E_{11}$, $\beta=\max\{n||C^{-1}A C||+1$, $~2\rho(A)+1\}$. The matrix $\w{B}$ is positive since $\beta>n||C^{-1}A C||$. It remains to prove that $\gen{B}=\gen{A}$. By the construction, $B$ equals its real Jordan normal form with the eigenvalue $\lambda_0+\beta$ in the position $(1,1)$. It has algebraic multiplicity $1$ due to $\beta>2\rho(A)$. By Lemma \ref{importantlemma}, there exists $g(x)\in\R[x]$ such that $g(B)=E_{11}$. Consequently,  $\gen{A}=\gen{B}$.
\end{proof}

Next we prove a lemma on algebras containing a diagonal matrix unit. It is important in order to obtain a special form of positively generated algebras in Theorem \ref{maintheorem} below.

\begin{Lemma} \label{structurallemma}
	Let $\A\subseteq M_n(\R)$ be a unital algebra, $E_{11}\in\A$. Then there exist a nonsingular matrix $C\in M_n(\R)$, integers $k_1,k_2,k_3\geq 0$, subalgebras $\B_j\subseteq M_{k_j}(\R)$ such that
	$$
	C^{-1}\A C=\left(\begin{array}{ccc}
	\B_1 & * & * \\
	O & \B_2 & *\\
	O& O& \B_3 
	\end{array}\right).
	$$
	If $k_j=0$ for some $j\in\{1,2,3\}$, then the corresponding block $\B_j$  is absent. Also one can find $j_0\in\{1,2,3\}$ with the conditions $\B_{j_0}=M_{k_{j_0}}(\R)$, $k_{j_0}\geq 1$. Moreover, $E_{ll}\in C^{-1}\A C$ for $l=1+\sum_{j=1}^{j_0-1}k_j$ (the sum over the empty set indicates $0$).
	
\end{Lemma}
\begin{proof}
	Consider $\A$ as an algebra of operators on $\R^n$.  Denote by $\{e_i\}_{i=1}^n$ the standard basis of $\R^n$, i.e. $(e_i)_j=\delta_{ij}$. Let $Pr_{\W}(\cdot)$ be the projection operator onto a subspace $\W$. Introduce
		$$\Psi_1=\{\V\subseteq\R^n~|~\V=\lin{\V},~ \A\V\subseteq\V,~e_1\in\V\},$$
		$$\Psi_2=\{\V\subseteq\R^n~|~\V=\lin{\V},~ \A\V\subseteq\V,~Pr_{\V}(e_1)=0\}.$$

We have $\R^n\in\Psi_1,\{0\}\notin\Psi_1$ and $\{0\}\in\Psi_2,\R^n\notin\Psi_2$. Note that $\Psi_1=\{\V\subseteq\R^n~|~\V=\lin{\V},~ \A\V\subseteq\V,~Pr_{\V}(e_1)\neq 0\}$ since $E_{11}\in\A$. So, $\Psi_1\cup\Psi_2$ contains all invariant subspaces of $\A$. Introduce

$$\Z_1=\bigcap\limits_{\V\in\Psi_1}\V=\min\limits_{\subseteq}~\Psi_1,$$
$$\Z_2=\sum\limits_{\V\in\Psi_2}\V=\max\limits_{\subseteq}~\Psi_2.$$

	There are four possibilities.
	\begin{enumerate}
		\item[Case 1.] $\Z_1\neq\R^n$, $\Z_1\cap\Z_2\neq\{0\}$. Let $k_1={\dim\Z_1\cap\Z_2}$,  $k_2=\dim\Z_1-{\dim\Z_1\cap\Z_2}$, $k_3=n-k_1-k_2$. Choose a basis $ {(}g_1, \ldots, g_{k_1}, g_{k_1+1},\ldots, \\g_{k_1+k_2}, g_{k_1+k_2+1},\ldots, g_n{)}$ of $\R^n$ where ${(}g_1,\ldots,g_{k_1}{)}$ is a basis of $\Z_1\cap\Z_2$, \\${(}g_1,\ldots,g_{k_1+k_2}{)}$ is a basis of $\Z_1$, $g_{k_1+1}=e_1$. In this basis
$$
\h{\A}=C^{-1}\A C=\left(\begin{array}{ccc}
\B_1 & * & * \\
O & \B_2 & *\\
O& O& \B_3 
\end{array}\right).
$$
Prove that $\B_2$ is irreducible. Assume the opposite. Let $\B_2$ have a non-trivial invariant subspace $\V'\subseteq\lin{e_{k_1+1},\ldots,e_{k_1+k_2}}$. Then $\V=\lin{g_1,\ldots, g_{k_1}}\oplus\V'$ is an invariant subspace of $\A$. Consider two possi\-bi\-li\-ties: either $Pr_{\V'}(g_{k+1})\neq0$ or $Pr_{\V'}(g_{k+1})=0$. Then we have a contradiction either with the minimality of $\Z_1$ or with $\lin{g_{k_1+1},\ldots,g_{k_1+k_2}}\cap\Z_2=\{0\}$, respectively. So, $\B_2$ is irreducible. Also the equalities $\col_{k+1}(C)=e_1$, $\col_1(C^{-1})=e_{k+1}$ imply $E_{(k+1)(k+1)}=C^{-1}E_{11}C\in\h{\A}$. Therefore $\B_2=M_{k_2}(\R)$ by Theorem \ref{burnside}.
	
	\item[Case 2.] $\Z_1\neq\R^n$, $\Z_2\cap \Z_1=\{0\}$. 
	Denote $k=\dim~\Z_1$. So one can choose a basis $(g_1,g_2,\ldots,g_{k},g_{k+1},\ldots,g_n)$ of $\R^n$ such that $g_1=e_1$, $(g_1,\ldots,g_{k})$ is a basis of $\Z_1$. In this basis $\A$ has the form
	$$\h{\A}=C^{-1}\A C = \left(\begin{array}{cc}
	\CC & * \\
	O & \B
	\end{array}\right).$$
	Note that $\CC$ is irreducible, otherwise we have a contradiction either with the minimality of $\Z_1$ or with $\Z_1\cap\Z_2=\{0\}$ (the reasoning is as in Case 1.). Theorem \ref{burnside} implies $\CC=M_{k}(\R)$ since $E_{11}\in \h{\A}$.

		\item[Case 3.] $\Z_1=\R^n$, $\Z_2\neq\{0\}$. 
Let $k=\dim\Z_2$. Choose a basis ${(}g_1,g_2,\ldots,g_{k},\\g_{k+1},\ldots,g_n{)}$ of $\R^n$ such that $g_{k+1}=e_1$, $(g_1,\ldots,g_{k})$ is a basis of $\Z_2$. In this basis 
$$\h{\A}=C^{-1}\A C = \left(\begin{array}{cc}
\B & * \\
O & \CC
\end{array}\right).$$
 Then $\CC$ is irreducible, otherwise we have a contradiction either with $\Z_1=\R^n$ or with the maximality of $\Z_2$ (the reasoning is as in Case 1.). Also $E_{(k+1)(k+1)}\in\h{\A}$. Therefore $\CC=M_{k_2}(\R)$ by Theorem \ref{burnside}.

	\item[Case 4.] $\Z_1=\R^n$, $\Z_2=\{0\}$. In this case $\A$ is irreducible and contains $E_{11}$. Hence $\A=M_n(\R)$ by Theorem \ref{burnside}.

	\end{enumerate} 
\end{proof}

\begin{Remark}
Lemma \ref{structurallemma} holds over an arbitrary field.
\end{Remark}

Here we provide a criterion for an algebra to be generated by positive matrices up to similarity. The previous lemma is applied to obtain some canonical form for such algebras. 

\begin{Theorem} \label{maintheorem}
	Let  $\A\subseteq M_n(\R)$ be a unital algebra. Then the following conditions are equivalent.
	\begin{enumerate}

\item $\A$ is a \PG-algebra up to similarity.
\item $\A$ contains a matrix with a simple real eigenvalue.
		\item $\A$ contains an idempotent matrix of rank $1$.

\item There exists a nonsingular matrix $C\in M_n(\R)$ such that $\h{\A}=C^{-1}\A C$ is one of the following types.

\begin{enumerate}
	\item $\h{\A}=M_n(\R)$.
	\item 		$\h{\A} = \left(\begin{array}{cc}
	M_k(\R) & * \\
	O & \B
	\end{array}\right),$ $E_{11}\in \h{\A}$.
	\item $\h{\A} = \left(\begin{array}{cc}
	\B & * \\
	O & M_k(\R)
	\end{array}\right),$ $E_{ll}\in \h{\A}$, $l=n-k+1$.
	\item $\h{\A}=\left(\begin{array}{ccc}
		\B_1 & * & * \\
		O & M_k(\R) & *\\
		O& O& \B_2 
	\end{array}\right)$, $\B_1\subseteq M_p(\R)$, $E_{ll}\in \h{\A}$, $l=p+1$.
\end{enumerate}
	\end{enumerate}
\end{Theorem}
\begin{proof}
	$ $
	\begin{enumerate} 
		\item[$1)\Rightarrow2)$] Apply Corollary \ref{simplecriterionpositive}. Perron's theorem ensures that a positive matrix has a positive eigenvalue with the strictly maximal absolute value (see \cite[p. 526, Theorem 8.2.8]{HornJ13}). This eigenvalue has algebraic mul\-ti\-pli\-ci\-ty~$1$.
		\item[$2)\Rightarrow3)$] Let $A$ be a matrix with a simple real eigenvalue $\lambda$. So, there exists a nonsingular matrix $C$ such that $C^{-1}AC$ equals a real Jordan normal form of $A$. Let $(C^{-1}AC)_{11}=\lambda$ without loss of generality. Applying Lemma \ref{importantlemma} there exists $f\in\R[x]$ such that $E_{11}=f(C^{-1}AC)=C^{-1}f(A)C$. Then $C E_{11}C^{-1}=f(A)\in\A$ is an idempotent of rank $1$.
		\item[$3)\Rightarrow 4)$] Let $E\in\A$ be an idempotent of rank $1$. Then there exists a real nonsingular matrix $C$ that turns $E$ to its Jordan normal form $C^{-1}EC=E_{11}$. It remains to apply Lemma \ref{structurallemma} to the algebra
		 $C^{-1}\A C$.
		\item[$4)\Rightarrow 1)$] Since $E_{ll}\in C^{-1}\A C$ for some $l$, we can find a nonsingular matrix $\h{C}$ such that $\h{C}^{-1}E_{ll}\h{C}>O$ by Lemma \ref{idempotentsimilar}. So, Corollary \ref{simplecriterionpositive} works.
	\end{enumerate}
\end{proof}

However, Theorem \ref{maintheorem} does not cover the algebras with a minimal positive generating system up to similarity. We provide the following example. 

\begin{Example}
	Let $\A\subseteq M_n(\R)$ be a subalgebra, $D_n(\R)\subseteq\A$. Then $\A$ is an \MPG-algebra up to similarity.  Indeed, let $H$ be an arbitrary positive matrix with distinct real eigenvalues. Choose a matrix $C$ such that $CHC^{-1}=D\in D_n(\R)$. If $\mgen(\A)=1$, then $\dim\A\leq n$ by the Cayley–Hamilton theorem. So, $\A=D_n(\R)$, $\gen{D}=\A$, $C^{-1}DC=H>O$. Assume that  $\mgen(\A)\geq 2$, then  $\mgen(\A)=2$ by Theorem \ref{IncGen}. Let $A\in\A$ be a matrix covering  $\A$. Choose $\alpha>0$ large enough that $\alpha H+ C^{-1}AC>O$ and $\Omega(\alpha D+ A)=\Omega(A)=\Omega(\A)$. Then $\gen{\{D,\alpha D+A\}}=\A$ by Theorem \ref{IncGen}. Also $C^{-1}DC=H>O$, $C^{-1}(\alpha D +A)C>O$.
\end{Example}

At the same time every unital matrix algebra can be made an \MPG-algebra if one considers the direct sum of it and the base field $\R$. This fact follows from the next lemma which is essential for the description of commutative \PG-algebras. 

\begin{Lemma} \label{oplus_R}
	Let $\A\subseteq M_n(\R)$ be a unital algebra, $n\geq 2$. Assume that there exists a nonsingular matrix $C$ such that $C^{-1}\A C= \R\oplus\B$ for some matrix algebra $\B\subseteq M_{n-1}(\R)$. Then $\A$ is an \MPG-algebra up to similarity.
\end{Lemma}
\begin{proof}
 Assume that $C=I$ without loss of generality. Let $\{B_1,\ldots,B_m\}$ be a minimal generating system of $\B$. Consider $A_1=[\lambda]\oplus B_1$, $\lambda\notin\sigma(B_1)$, $A_i=[0]\oplus B_i$ for $i=2,\ldots,m$. So, $\gen{A_1,\ldots,A_m}=\A$ by Lemma \ref{importantlemma}. Since $m\geq\mgen(\A)\geq\mgen(\B)=m$, we obtain that $\{A_1,\ldots,A_m\}$ is a minimal generating system of $\A$. Applying Lemma \ref{onelemma} one finds $\w{A}_1\in\gen{A_1}$ such that $\gen{\w{A}_1}=\gen{A_1}$ and $\w{A}_1$ is similar to a positive matrix. Thus, $\gen{\w{A}_1,A_2\ldots,A_m}=\gen{A_1,\ldots,A_m}=\A$. There exists $S$ with ${\h{A}_1=S^{-1}\w{A}_1S>O}$. Denote $\h{A}_i=S^{-1}A_iS$ for $i=2,\ldots, m$. Then $\{\h{A}_1\}\cup\{\h{A}_1\frac{||\h{A}_i||+1}{||\h{A}_1||}+\h{A}_i\}_{i=2}^m$ is a positive minimal generating system of $S^{-1}\A S$.
\end{proof}

The following theorem is obtained by applying Theorem \ref{maintheorem} and Lemma \ref{oplus_R} to commutative matrix algebras.

\begin{Theorem} \label{commutativematrices}
	Let $\A\subseteq M_n(\R)$ be a commutative unital matrix algebra, $n\geq 2$. Then the following conditions are equivalent.
	\begin{enumerate}
		\item $\A$ is a \PG-algebra up to similarity.
		\item $\A$ is an \MPG-algebra up to similarity.
		\item There exists a nonsingular matrix $C$ that
		$C^{-1}\A C= \R\oplus\B$ for some matrix algebra $\B\subseteq M_{n-1}(\R)$.
	\end{enumerate}
\end{Theorem}
\begin{proof}
		The equivalence of items $1$ and $3$ follows from item $4$ of Theorem \ref{maintheorem}. Item $3$ implies item $2$ by Lemma \ref{oplus_R}. Also item $1$ clearly follows from item $2$. 
\end{proof}
 
 From the above theorem we immediately obtain the next corollary about one-generated algebras.
 
\begin{Corollary} \label{mpg-one}
	Consider an arbitrary real matrix $A\in M_n(\R)$. Then $\gen{A}$ is a \PG-algebra up to similarity if and only if $\gen{A}$ is an \MPG-algebra up to similarity. Moreover, it is equivalent to the requiring that $A$ has a simple real eigenvalue.
\end{Corollary}

Now we turn to simple and semisimple algebras. First prove the following technical lemma.

\begin{Lemma} \label{semisimplelemma}
	Let $\A\subseteq M_n(\R)$ be a simple subalgebra. Assume that each matrix $X$ of $\A$ has the following block form
	$$X=\left(\begin{array}{ccc}
	O_{k_1} & X_{12} & X_{13} \\
	O & X_{22} & X_{23}\\
	O& O& O_{k_3} 
	\end{array}\right),$$
	where $X_{22}\in M_{k_2}(\R)$, $k_2\geq 1$, $k_1,k_3\geq 0$, $k_1+k_2+k_3=n$. Equalities $k_1=0$, $k_2=0$ indicate that the corresponding blocks are absent. Let $E_*=O_{k_1}\oplus I_{k_2}\oplus O_{k_3}$, $E_*\A E_*=\{E_*XE_*~|~X\in\A\}=M_{k_2}(\R)$, $E_{ll}\in\A$, $l=k_1+1$. Consider a matrix $E$ of the type
	$$E=\left(\begin{array}{ccc}
	O_{k_1} & Y_{12} & Y_{13} \\
	O & I_{k_2} & Y_{23}\\
	O& O& O_{k_3} 
	\end{array}\right).$$
	Assume that $E\in\A$ and $EX=XE=X$ for all $X\in\A$. Then $\A=O_{k_1}\oplus M_{k_2}(\R)\oplus O_{k_3}$ and $E=E_*$.
\end{Lemma}
\begin{proof}
	Consider the relation $EX=XE=X$
	$$\left(\begin{array}{ccc}
	O_{k_1} & Y_{12}X_{22} & Y_{12}X_{23} \\
	O & X_{22} & X_{23}\\
	O& O& O_{k_3} 
	\end{array}\right)=
	\left(\begin{array}{ccc}
	O_{k_1} & X_{12} & X_{12}Y_{23} \\
	O & X_{22} & X_{22}Y_{23}\\
	O& O& O_{k_3} 
	\end{array}\right)=
	\left(\begin{array}{ccc}
	O_{k_1} & X_{12} & X_{13} \\
	O & X_{22} & X_{23}\\
	O& O& O_{k_3} 
	\end{array}\right).$$
	Therefore
	\begin{equation} \label{eq1}
	\A=\left\lbrace \left.\left(\begin{array}{ccc}
	O_{k_1} & Y_{12}X_{22} & Y_{12}X_{22}Y_{23} \\
	O & X_{22} & X_{22}Y_{23}\\
	O& O& O_{k_3} 
	\end{array}\right)~  \right| ~ X_{22}\in M_{k_2}(\R) \right\rbrace .
	\end{equation}
	Introduce the following two-sided ideals:
	$$I_1=\{X\in\A~|~Y_{12}X_{22}=O\}\vartriangleleft\A,~I_2=\{X\in\A~|~X_{22}Y_{23}=O\}\vartriangleleft\A.$$
	Then $I_1,I_2\neq\{O\}$ since $E_{ll}\in I_1\cap I_2$. So, $I_1=I_2=M_{k_2}(\R)$ because $\A$ is simple. It indicates that $Y_{12}X_{22}=O$ and $X_{22}Y_{23}=O$ for any $ X_{22}\in M_{k_2}(\R)$. This is obviously equivalent to $Y_{12}=O$, $Y_{23}=O$. Now all follows from Equality \ref{eq1}.
\end{proof}

The following theorem describes semisimple algebras which have positive generators up to similarity.

\begin{Theorem} \label{semis_theorem}
	 Let $\A\subseteq M_n(\R)$ be a unital \PG-algebra up to similarity, $n\geq 2$.
	 \begin{enumerate}
	 	\item $\A$ is simple if and only if $\A=M_n(\R)$.
	 	\item $\A$ is semisimple but not simple if and only if there exists a nonsingular $C\in M_n(\R)$ such that $C^{-1}\A C=M_k(\R)\oplus\B$ for $k\geq 1$ and some semisimple matrix algebra $\B\subseteq M_{n-k}(\R)$.
	 \end{enumerate}
\end{Theorem}
\begin{proof} Theorem \ref{maintheorem} ensures that $\A$ is similar to an algebra that has one of types $(a)-(d)$. 
	\begin{enumerate}
		\item Assume the opposite. So, possible types are $(b)$, $(c)$, $(d)$. Consider $(d)$, other cases are proved similarly. Since $\A$ is simple, the two-sided ideal generated by $E_{ll}$ is trivial, i.e. $I_n\in \A E_{ll}\A$. Note that for any $X\in \A$, $XE_{ll}$ and $E_{ll}X$ have zeros in the first and the third diagonal blocks. So, $I_n\notin \A E_{ll}\A$, this is contradiction.
		\item Again we examine the case $(d)$ only, others are considered in the same manner. Since $\A$ is semisimple, there exists a family of central idempotents $\{E_{i}\}_{i=1}^N\subseteq\A$ such that $E_iE_j=O$ if $i\neq j$, $\sum_{i=1}^N E_i=I_n$, $E_i\A=\A E_i$ is a simple algebra with the identity $E_i$. There exists $E_i$ with non-zero second diagonal block due to $\sum_{i=1}^N E_i=I_n$. Let it be $E_1$ without loss of generality. Denote
		$$E_1=\left(\begin{array}{ccc}
	Y_{11} & Y_{12} & Y_{13} \\
	O & Y_{22} & Y_{23}\\
	O& O& Y_{33} 
	\end{array}\right),$$
	where $Y_{ii}\in M_{k_i}(\R)$, $i\in\{1,2,3\}$, $k_2\geq 1$, $k_1+k_2+k_3=n$. Since $E_1$ is in the center of $\A$, $Y_{22}$ must be in the center of $M_{k_2}(\R)$, i.e. $Y_{22}=I_{k_2}$. Note that $E_1\in (E_1\A) E_{ll} (E_1\A)$ by virtue of simplicity of $E_1\A$. Hence $Y_{11}=O$, $Y_{33}=O$. Now all follows from Lemma \ref{semisimplelemma} applied to the algebra $E_1\A$ and the matrix $E_1$.
	\end{enumerate}
\end{proof}

\section{Nonnegative generators up to similarity}\label{4}

This section deals with nonnegatively generated matrix algebras up to similarity. The nonnegative case is more difficult than positive one. However, we, in some sense, generalize the methods of the previous section for algebras of a block diagonal type up to similarity.

\subsection{Some examples} \label{4.1}

This subsection provides examples that give answers to some natural questions about nonnegatively generated algebras up to similarity. First we prove the following proposition.

\begin{Proposition} \label{sumnonnegative}
	Let $\A\subseteq\bigoplus_{i=1}^k M_{n_i}(\R)$ be a unital algebra, $k\geq 1$, $n_i\geq 1$. Assume that there exists a matrix $E\in\A\setminus\{O\}$ which has the projection onto each block is either zero or an idempotent matrix of rank $1$. Then $\A$ is nonnegatively generated up to similarity.
\end{Proposition}
\begin{proof}
Denote by $E_i\in M_{n_i}(\R)$ the projection of $E$ onto the $i$-th block, so $E=\bigoplus_{i=1}^k E_i$. Let $E_i$ be an idempotent of rank $1$  for $i= 1,\ldots,m$ and be zero for $i=m+1,\ldots, k$. Reasoning as in the item $[3)\Rightarrow 4)]$ of the proof of Theorem \ref{maintheorem} introduce matrices $C_i\in M_{n_i}(\R)$ for $i=1,\ldots,m-1$ such that $C_i^{-1}E_iC_i>0$. Let $l=\sum_{i=m+1}^k n_i$, $\h{E}=E_m\oplus O_{l}$. So, $\h{E}$ is an idempotent of rank $1$. Again we find a nonsingular $\h{C}\in M_{(l+n_m)}(\R)$ such that $\h{C}^{-1}\h{E}\h{C}>0$. Then $S=\left( \bigoplus_{i=1}^{m-1}C_i\right) \oplus\h{C}$ is the desired transition matrix since $S^{-1}ES$ is nonnegative and covers $S^{-1}\A S$.
\end{proof}

The next example shows that there exist algebras being nonnegatively generated up to similarity but not positively generated up to similarity.  Moreover, the example demonstrates that a \PG-algebra up to similarity is not the same as a \PG-algebra up to an arbitrary $\R$-algebra isomorphism.
\begin{Example} \label{example1}
	Let $\A\subseteq M_n(\R)$ be a \PG-algebra up to similarity. Consider the algebra $\B=\{\underbrace{A\oplus A\oplus\ldots\oplus A}_{k~\text{times}}~|~A\in \A\}\subseteq M_{nk}(\R)$ for some ${k\geq 2}$. Then $\A$ and $\B$ are clearly isomorphic as $\R$-algebras. However, Theorem \ref{maintheorem} ensures that $\B$ is not a \PG-algebra up to similarity. Nevertheless, Proposition \ref{sumnonnegative} shows the algebra $\B$ to be nonnegatively generated up to similarity.
\end{Example}

Now we provide examples of algebras which are not nonnegatively ge\-ne\-ra\-ted up to similarity. By the way we obtain algebras that are not nonnegatively generated but nonnegatively generated up to similarity.

\begin{Example}
 	Let $\CC=\left. \left\lbrace \left(\begin{array}{cc}
 	a & b \\
 	-b & a
 \end{array}\right)~\right| ~a,b\in\R\right\rbrace \cong\C$. If $A\in\CC\setminus\R I_2$, then $\sigma(A)\cap\R=\varnothing$. So, for all nonsingular matrices $C\in M_2(\R), $ $C^{-1}\CC C$  does not contain a nonnegative matrix except scalar ones (\cite[p. 529, Theorem 8.3.1]{HornJ13}). We conclude from this that $\CC$ has not nonnegative generators up to similarity. Consider an algebra $\A=\R I_k\oplus\CC$ for $k\geq 2$. It does not have a nonnegative generating system. However, $\A\subseteq \underbrace{\R\oplus\R\oplus\ldots\oplus\R}_{k-1~\text{times}}\oplus M_3(\R)$, $\underbrace{[1]\oplus\ldots\oplus[1]}_{k-1~\text{times}}\oplus([1]\oplus O_2)\in\A$, and $\A$ is nonnegatively generated up to similarity by Proposition \ref{sumnonnegative}. We may repeat all previous arguments in the case $\CC\cong\HH$ where $\HH$ is Hamilton's quaternion algebra.
\end{Example}

\subsection{Direct sums of algebras} \label{4.2}

Here we prove that an arbitrary unital matrix algebra becomes nonnegatively generated up to similarity if one considers a direct sum of it and another nonnegatively generated algebra. Moreover, the property to have a minimal nonnegative generating system is also saved. These results (Theorem \ref{direct_sum_theorem}) will imply some new sufficient conditions to be nonnegatively generated up to similarity (Corollaries \ref{center}, \ref{centralizer}). First we prove a technical lemma.

\begin{Lemma} \label{horrortables}
	Let $A=O_{k_1}\oplus T\in M_n(\R)$, where $T\in M_{k_2}(\R)$, $k_1+k_2=n$, $k_2\geq 2$, $k_1\geq 1$. Assume that $\row_1(T)=(t_1,t_2,\ldots,t_{k_2})$, $t_1>0$, $\col_1(T)=(t_1,e_2,\ldots,e_{k_2})^t$, $C=\KKK_{k_1+1}\oplus I_{k_2-1}$. Consider $$B=C^{-1}AC=\left(\begin{array}{cc}
	B_1 & B_{12} \\
	B_{21} & B_2
	\end{array}\right)$$
	where $B_1\in M_{k_1+1}(\R)$, $B_2\in M_{k_2-1}(\R)$. Then next conditions are fulfilled.
	\begin{enumerate}
		\item $B_1>0$, $||B_1||=\frac{t_1}{k_1+1}$.
		\item $\row_q(B_{21})=\frac{1}{k_1+1}(e_{q+1},\ldots,e_{q+1})$, $q=1,\ldots,k_2-1$.
		\item $\col_l(B_{12})=(t_{l+1},\ldots,t_{l+1})^t$, $l=1,\ldots,{k_2-1}$.
		\item $(B_2)_{ij}=(T)_{(i+1)(j+1)}$, $i,j=1,\ldots,k_2-1$.
	\end{enumerate}
In particular, if $T\geq O$, then $B\geq O$. Also if $T>O$, then $B>O$.
\end{Lemma}
\begin{proof}
Let $\h{k}_1=k_1+1$, $\h{k}_2=k_2-1$. Apply Lemma \ref{idempotentsimilar} and Notation \ref{onlynotation}.  
$$
B={\footnotesize
	\renewcommand{\arraystretch}{0.9} 
	\renewcommand{\tabcolsep}{0.1cm}   
\left(\begin{tabular}{cccc|ccc}
\multicolumn{3}{r}{\multirow{3}{*}{\large $\KKK^{-1}_{\h{k}_1}$}} & $~$ & \multicolumn{3}{c}{\multirow{3}{*}{\Large $O$}}\\
\multicolumn{3}{r}{} &~& 	\multicolumn{3}{c}{}\\
\multicolumn{3}{r}{} &~& 	\multicolumn{3}{c}{}\\
$~$ & $~~~$ & $~$ & $~~$ & $~~$& $~~~$ & $~~~$\\
\hline
~ &~~~ &~& $~~$&\multicolumn{3}{|c}{\multirow{3}{*}{\large $I_{\h{k}_2}$}}\\
\multicolumn{3}{c}{\multirow{2}{*}{\Large $O$}}& &  \multicolumn{3}{|c}{} \\
\multicolumn{3}{c}{} &  & \multicolumn{3}{|c}{}\\
\end{tabular}\right)
	\left(\begin{tabular}{cccc|ccc}
	\multicolumn{3}{r}{\multirow{3}{*}{\large $O_{k_1}$}} & $0$ & \multicolumn{3}{c}{\multirow{3}{*}{\Large $O$}}\\
	\multicolumn{3}{r}{} &{\tiny \vdots}& 	\multicolumn{3}{c}{}\\
	\multicolumn{3}{r}{} &0& 	\multicolumn{3}{c}{}\\
	$0$ & $\ldots$ & $0$ & $t_1$ & $t_2$& $\ldots$ & $t_{k_2}$\\
	\hline
	0 &\ldots &0& $e_2$&\multicolumn{3}{|c}{\multirow{3}{*}{\large $B_2$}}\\
	\multicolumn{3}{c}{\multirow{2}{*}{\Large $O$}}&{\tiny \vdots}&  \multicolumn{3}{|c}{} \\
	\multicolumn{3}{c}{} & $e_{k_2}$ & \multicolumn{3}{|c}{}\\
	\end{tabular}\right)
	\left(\begin{tabular}{cccc|ccc}
\multicolumn{3}{r}{\multirow{3}{*}{\large $\KKK_{\h{k}_1}$}} & $~$ & \multicolumn{3}{c}{\multirow{3}{*}{\Large $O$}}\\
\multicolumn{3}{r}{} &~& 	\multicolumn{3}{c}{}\\
\multicolumn{3}{r}{} &~& 	\multicolumn{3}{c}{}\\
$~$ & $~~~$ & $~$ & $~~$ & $~~$& $~~~$ & $~~~$\\
\hline
~ &~~~ &~& $~~$&\multicolumn{3}{|c}{\multirow{3}{*}{\large $I_{\h{k}_2}$}}\\
\multicolumn{3}{c}{\multirow{2}{*}{\Large $O$}}& &  \multicolumn{3}{|c}{} \\
\multicolumn{3}{c}{} &  & \multicolumn{3}{|c}{}\\
\end{tabular}\right)=
}
$$

$$=
{
	\renewcommand{\arraystretch}{0.9} 
	\renewcommand{\tabcolsep}{0.1cm}   
	{\footnotesize\left(\begin{tabular}{cccc|ccc}
\multicolumn{4}{c|}{\multirow{3}{*}{\large $\alpha\1_{k_1+1}$}}& $t_2$ & $\ldots$ & $t_{k_2}$\\
\multicolumn{4}{c|}{} &  $\vdots$ & $\vdots$ & $\vdots$\\
\multicolumn{4}{c|}{} &$t_2$ & $\ldots$ & $t_{k_2}$\\
\multicolumn{4}{c|}{} & $t_2$ & $\ldots$ & $t_{k_2}$\\
\hline
$p_2$ &\ldots &$p_2$& $p_2$&\multicolumn{3}{|c}{\multirow{3}{*}{\large $B_2$}}\\
$\vdots$& \ldots& \vdots&{ \vdots}&  \multicolumn{3}{|c}{} \\
$p_{k_2}$ &\ldots &$p_{k_2}$& $p_{k_2}$ & \multicolumn{3}{|c}{}\\
\end{tabular}\right)}
},~\alpha=\dfrac{t_1}{k_1+1}>0,~p_i=\dfrac{e_i}{k_1+1}.
$$
\end{proof}

\begin{Theorem} \label{direct_sum_theorem}
	Let $\A\subseteq M_{k_1}(\R)$, $\B\subseteq M_{k_2}(\R)$ be unital matrix algebras, $k_1,k_2\geq 1$.
	\begin{enumerate}
		\item If at least one of algebras $\A$, $\B$ is nonnegatively or positively generated up to similarity then the algebra $\A\oplus\B$ has the same property.
		\item If at least one of algebras $\A$, $\B$ has a minimal nonnegative or positive generating system up to similarity then the algebra $\A\oplus\B$ possesses the same property.
	\end{enumerate}
\end{Theorem}
\begin{proof}
 Since algebras $\A\oplus\B$ and $\B\oplus \A$ are similar, it is enough to consider only the case $\A\oplus\B$ when $\B$ has the indicated properties.
\begin{enumerate}
	\item If $k_2=1$, then all follows from item $4$ of Theorem \ref{maintheorem}. So, let $k_2\geq 2$. One may find $S\in M_{k_2}(\R)$ such that $S^{-1}\B S$ has a nonnegative covering matrix (see item $1$ of Theorem \ref{Criterion}). Conjugating by the matrix $I_{k_1}\oplus S$ we can assume $S=I_{k_2}$, i.e. there exists $T\in\B$ such that $T\geq O$, $\Omega(\B)=\Omega(T)$. Since $I_{k_2}\in\B$, then $(T)_{11}>0$. Let $C=\KKK_{k_1+1}\oplus I_{k_2-1}$, $\w{T}=O_{k_1}\oplus T$. Applying Lemma \ref{horrortables} to matrices of the algebra $O_{k_1}\oplus\B$ we obtain that $C^{-1}\w{T}C\geq O$ and $\Omega(C^{-1}(O_{k_1}\oplus\B) C)=\Omega(C^{-1}\w{T}C)$. Also if $T>O$, then $C^{-1}\w{T}C> O$. Note that
	$$\Omega(C^{-1}(\A\oplus\B)C)\subseteq \Omega(C^{-1}(\A\oplus O_{k_2})C)\cup \Omega(C^{-1}(O_{k_1}\oplus \B)C).$$
	Since the first block of the matrix $C$ has the size $k_1+1>k_1$,
 $$\Omega(C^{-1}(\A\oplus O_{k_2})C)\subseteq\{(i,j)~|~i,j=1,\ldots,k_1+1\}\subseteq \Omega(C^{-1}\w{T}C).$$
 So, $\Omega(C^{-1}\w{T}C)=\Omega(C^{-1}(\A\oplus\B)C)$.
 \item If $k_2=1$, then see Lemma \ref{oplus_R}. Let $k_2\geq 2$. We may assume that $\B$ has a minimal nonnegative generating system $\{Z_1,\ldots, Z_m\}$ without loss of generality. Let $\{A_i\oplus B_i\}_{i=1}^l$ be a minimal generating system of $\A\oplus\B$. Obviously, $l\geq m$. If $l>m$, denote $Z_{m+1}=O_{k_2},\ldots, Z_{l}=O_{k_2}$. Put $C=\KKK_{k_1+1}\oplus I_{k_2-1}$ again. Consider $U_i=A_i\oplus(z_iI_{k_2}+Z_i)$, $i=1,\ldots, l$ with
 $$z_i=\max\{(k_1+1)||\KKK_{k_1+1}^{-1} (A_{i}\oplus[0]) \KKK_{k_1+1}||+1, 2\rho(A_i\oplus B_i)+1\}.$$
\begin{enumerate}
	\item Since $z_i> 2\rho(A_i\oplus B_i)$, we have $\sigma(A_i)\cap\sigma(z_iI_{k_2}+Z_i)=\varnothing$. By Lemma \ref{importantlemma}, there exists $f_i(x)\in\R[x]$ such that $f_i(U_i)=O_{k_1}\oplus I_{k_2}$. Therefore $A_i\oplus O_{k_2},O_{k_1}\oplus Z_i\in\gen{U_1,\ldots, U_l}$ for $i=1,\ldots,l$. So, $\gen{U_1,\ldots, U_l}=\A\oplus\B$. 
	\item For $i=1,\ldots,l$, we have $$C^{-1}U_iC=C^{-1}(A_i\oplus O_{k_2})C + C^{-1}(O_{k_1}\oplus(z_iI_{k_2}+Z_i))C=$$ $$\KKK^{-1}_{k_1+1}(A_i\oplus[0])\KKK_{k_1+1}\oplus O_{k_2-1} + C^{-1}(O_{k_1}\oplus(z_iI_{k_2}+Z_i))C.$$ Lemma \ref{horrortables}, the condition $Z_{i}\geq O$, and the inequality $$z_i>(k_1+1)||\KKK_{k_1+1}^{-1} (A_{i}\oplus[0]) \KKK_{k_1+1}||-(Z_i)_{11}$$ imply $C^{-1}U_iC\geq O$. Moreover, if $Z_i>O$ then $C^{-1}U_iC>O$.
\end{enumerate}
\end{enumerate}
\end{proof}

Now we apply the previous theorem to the centralizer (commutant) of a matrix with at least one real eigenvalue.

\begin{Corollary} \label{centralizer}
	Let $A\in M_n(\R)$, $\sigma(A)\cap\R\neq\varnothing$. Then the centralizer $C(A)$ is nonnegatively generated algebra up to similarity.
\end{Corollary}
\begin{proof}
	Without loss of generality, $A$ equals its real Jordan normal form. Consider an arbitrary $\lambda\in\sigma(A)\cap\R$. There are two cases.
\begin{enumerate}
	\item Let $\sigma(A)=\{\lambda\}$. Then $A=\bigoplus_{i=1}^kJ_{n_i}(\lambda)$. The centralizer $C(A)$ is of known structure. Each $R\in C(A)$ has a block form in accordance with the partition $A=\bigoplus_{i=1}^kJ_{n_i}(\lambda)$ (see \cite[Chapter VIII, \textsection $2$]{Gant60} or \cite[Chapter VII, 7.01]{Wed34}). Exactly for all $R\in C(A)$, we have $R=(R)_{\alpha\beta}$ with $\alpha,\beta=1,\ldots,k$ and $(R)_{\alpha\beta}=R_{n_\alpha n_\beta}(s^{\alpha\beta}_{1}, \ldots, s^{\alpha\beta}_{\min\{n_\alpha,n_\beta\}})$ has a regular triangular form. Conversely, for all real tuples $(s^{\alpha\beta}_{1}, \ldots, s^{\alpha\beta}_{\min\{n_\alpha,n_\beta\}})$, the matrices of the indicated form are included in $C(A)$. Consider $\h{R}\in C(A)$ such that for any $\alpha,\beta$, $(\h{R})_{\alpha\beta}=R_{n_\alpha n_\beta}(1, \ldots, 1)$. This matrix is obviously nonnegative and covers $C(A)$. Item $1$ of Theorem~\ref{Criterion} works.
	\item Let the cardinality $|\sigma(A)|>1$. We may assume that $A=P\oplus Q$ where $Q=\bigoplus_{i=1}^kJ_{n_i}(\lambda)$ is the direct sum of all Jordan cells corresponding to $\lambda$. Since $\sigma(P)\cap\sigma(Q)=\varnothing$, we have
	$C(A)=C(P)\oplus C(Q)$ \cite[Chapter VIII, \textsection $2$]{Gant60}.  It remains to apply the previous item and Theorem \ref{direct_sum_theorem}.
\end{enumerate}
\end{proof}

Now Theorem \ref{direct_sum_theorem} is applied to algebras with non-trivial center under a special condition.

\begin{Corollary} \label{center}
	Let $\A\subseteq M_n(\R)$ be a unital algebra. Assume that the center of $\A$ contains a matrix $A$ with the following property: there exists $\lambda\in\sigma(A)\cap\R$ of geometric multiplicity $1$. Then $\A$ is nonnegatively generated up to similarity.
\end{Corollary}
\begin{proof}
	Let $\lambda=1$, $A$ equals its real Jordan normal form without loss of generality. Since eigenvalue $1$ has geometric multiplicity $1$, $A$ contains exactly one Jordan cell corresponding to it. If $A=J_n(1)$, then all matrices from $\A$ have a regular triangular form (see \cite[Chapter VIII, \textsection $2$]{Gant60} or \cite[Chapter VII, 7.01]{Wed34}). So, $A^{n-1}$ is a nonnegative covering matrix. Let $A=P\oplus J_k(1)$, $1\notin\sigma(P)$, $k\geq 1$, $n-k\geq 1$. If $k=1$, all follows from Lemma~\ref{onelemma}. Let $k\geq 2$. Note that each $B\in\A$ has the form $B=X\oplus Y$ for some $Y\in M_k(\R)$ and $X\in M_{n-k}(\R)$ because $A$ is in the center of $\A$ and $1\notin\sigma(P)$ \cite[Chapter VIII, \textsection $2$]{Gant60}. Also $Y\in C(J_k(1))$, $Y$ has a regular triangular form. Applying Lemma~\ref{importantlemma} we find $f(x)\in\R[x]$ that $f(A)=O_{n-k}\oplus I_k$. Hence, $\A=\A_1\oplus\A_2$ and $J^{k-1}_k(1)$ is a nonnegative matrix covering $\A_2$. It remains to apply Theorem~\ref{direct_sum_theorem}. 
\end{proof}

The last corollary of Theorem \ref{direct_sum_theorem} is on one-generated algebras.

\begin{Corollary} \label{onegennonnegth}
	Let $A\in M_n(\R)$, $\sigma(A)\cap\R\neq\varnothing$. Then $\gen{A}$ is generated by one nonnegative matrix up to similarity.
\end{Corollary}
\begin{proof}
Without loss of generality, $A$ equals its real Jordan normal form and $0\in\sigma(A)\cap\R$. Consider a direct sum of all cells with the eigenvalue $0$: $Q=\bigoplus_{i=1}^kJ_{n_i}(0)$, $m=\sum_{i=1}^k n_i$. If $A=Q$, then all is proved. Let $A=P\oplus Q$, $0\notin\sigma(P)$. Applying Lemma \ref{importantlemma}, we find $f(x)\in\R[x]$ with $f(A)=O_{(n-m)}\oplus I_m$. Therefore $\gen{A}=\gen{P}\oplus\gen{Q}$ and all follows from item $2$ of Theorem~\ref{direct_sum_theorem}. 
\end{proof}

\begin{Problem}
	Find necessary and sufficient conditions for a given unital matrix algebra $\A$ to be similar to $(1)$ a nonnegatively generated algebra, $(2)$ a minimally nonnegatively generated algebra, $(3)$ an \MPG-algebra.
\end{Problem}

\section{Nonnegative semi-commuting generating systems of matrix incidence algebras}\label{5}

Here we find all possible dimensions of algebras generated by two nonnegative semi-commuting matrices. First show that a real matrix incidence algebra has such generators.

\begin{Theorem} \label{theorem_1}
	Each real matrix incidence algebra $\A$ is generated by two nonnegative semi-commuting matrices.
\end{Theorem}
\begin{proof}
	First assume that $\A\subseteq T_n(\R)$. Let $A,D\in \A$, $A$ be a nonnegative matrix covering $\A$, $D = \diag\{d_1,d_2,\ldots,d_n\}$ with $d_1> d_2>\ldots > d_n> 0$. Then 
	$$[D, A] = DA - AD =\sum\limits_{i\leq j}d_ia_{ij}E_{ij} - \sum\limits_{i\leq j}d_ja_{ij}E_{ij}=\sum\limits_{i\leq j} (d_i - d_j)a_{ij}E_{ij}\geq 0.$$
	So, $A,D$ semi-commute. Also they generate the algebra $\A$ by virtue of Theorem \ref{IncGen}.
	
	Let the incidence algebra $\A$ be not upper-triangular. Theorem \ref{iso} ensures that there exists a permutation matrix $P$ such that $\w{\A}=P^{-1}\A P$ is a triangular incidence algebra. Let $\w{A},\w{B}$ be semi-commuting nonnegative matrices that generate $\w{\A}$. Consider $A=P\w{A}P^{-1}$, $B=P\w{B}P^{-1}$. Then $[A,B]=P[\w{A},\w{B}]P^{-1}$. Besides, $\gen{\{A,B\}}=\A$. Since $P$ is a permutation matrix, we have $A,B,[A,B]\geq O$. 
\end{proof}

Now we determine the dimensions which matrix incidence algebras may have.

\begin{Theorem} \label{theorem_2}
	Fix natural $n\geq 2$. Then for each $k\in \mathbb{N}$ with $n\leq k\leq\frac{n(n+1)}{2}$, there exists a matrix incidence algebra $\A$ of the dimension $k$.
	
\end{Theorem}
\begin{proof} Consider matrix algebras of the following type
	
	$${
	\left(\begin{array}{ccccc}
	*& \ldots & \ldots & \ldots & \ldots \\
	0 & * & 6 & 5 & 4 \\
	0 & 0 & * & 3 & 2 \\
	0 & 0 & 0 & * & 1 \\
	0 & 0 & 0 & 0 & * 
	\end{array}\right).}
	$$
	
	Any real numbers can be instead of $*$ and in the positions $1,2,3,\ldots,k-n$. Zeroes are elsewhere. In other words, the algebra includes exactly matrices of the type
	$$
	\left(\begin{array}{cc}
	D_{n-m-1} & O \\
	O & X_{m+1}
	\end{array}\right)
	$$
	where $D_{n-m-1}$ is a square diagonal matrix of size $n-m-1$, $X_{m+1}$ is ${(m+1)\times(m+1)}$ matrix. Also
	
	$$X_{m+1}=
	\left(\begin{array}{cc}
	\alpha & c^{T} \\
	O_{m\times 1} & Y_{m}
	\end{array}\right)
	$$
	with an upper-triangular matrix $Y_{m}$, a real scalar $\alpha$, and a real vector $~c^{T} = (0, \ldots, 0,c^{l}, \ldots,c^{m}).$ The numbers $m,~l$ are uniquely determined by the parameter $k$. 
	
	It is necessary to show that the algebra was constructed correctly. It is enough to explain why the zeroes are saved in the pointed positions after the matrix multiplication
	$$
	\left(\begin{array}{cc}
	D_{n-m-1} & O \\
	O & X_{m+1}
	\end{array}\right)\cdot
	\left(\begin{array}{cc}
	\widetilde{D}_{n-m-1} & O \\
	O & \widetilde{X}_{m+1}
	\end{array}\right).
	$$
	The zeroes are saved out of blocks and in the diagonal $(n-m-1)\times(n-m-1)$ matrices. Consider
	$$X_{m+1}\widetilde{X}_{m+1}=
	\left(\begin{array}{cc}
	\alpha & c^{T} \\
	O_{m\times 1} & Y_{m}
	\end{array}\right)\cdot
	\left(\begin{array}{cc}
	\widetilde{\alpha} & \widetilde{c}^{~T} \\
	O_{m\times 1} & \widetilde{Y}_{m}
	\end{array}\right) = 
	\left(\begin{array}{cc}
	\alpha\widetilde{\alpha} & \alpha\widetilde{c}^{~T}+ c^{T}\widetilde{Y}_m \\
	O_{m\times 1} & Y_{m}\widetilde{Y}_{m}
	\end{array}\right).
	$$
	
	A product of two upper-triangular matrices  $Y_m\widetilde{Y}_m$ is upper-triangular, therefore the zeroes are saved here. 
	Let $\widetilde{y}_{i}$ be the $i$-th column of the matrix $\widetilde{Y}_m$. Since the matrix is upper-triangular, then $\widetilde{y}^{~T}_{i} = (\widetilde{y}^{~1}_{i}, \widetilde{y}^{~2}_{i}, \ldots, \widetilde{y}^{~i}_{i}, 0, \ldots, 0).$
	We have	
	$$\alpha\widetilde{c}^{~T}+ c^{T}\widetilde{Y}_m = (0,\ldots, 0, \alpha\widetilde{c}^{~l},\ldots, \alpha\widetilde{c}^{~m}) + (<c, \widetilde{y}_{1}>, \ldots, <c, \widetilde{y}_{m}>)$$
	where $<\cdot,\cdot>$ is the Euclidean scalar product. Then for each $i<l$, ${<c, \widetilde{y}_{i}>=0}$. Hence the zeroes are also preserved here.
	
	Thus, algebra is designed correctly. It has the basis of $k$ matrix units which fulfils the definition of matrix incidence algebra.  
\end{proof}

Theorems \ref{theorem_1}, \ref{theorem_2} imply the following corollary which provides the solution to Problem \ref{problem}.

\begin{Corollary} \label{solution}
	Fix natural $n\geq 2$. Then for any $k\in \mathbb{N}$ with $n\leq k\leq\frac{n(n+1)}{2}$ there exist matrices $A,B \in M_n(\R)$ such that $A, B\geq 0$, $[A, B]\geq 0$, $\dim\gen{\{A, B\}} = k$.
\end{Corollary}

\begin{center}
{\bf Acknowledgements}
\end{center}
The author is profoundly grateful to his supervisor Prof. O.V. Markova for posing the problems and helpful discussions, to Prof. A.E. Guterman for useful comments.

\end{document}